\def\dom{{\rm Dom}}

\def\={=\!\!}
\newcommand*\dep{{=\mkern-1.2mu}}
\newcommand*\depp[1]{{=\mkern-1.2mu}_{#1}\mkern-1.4mu}

\documentclass[12pt]{article}
\usepackage{amsthm} 
\usepackage{graphicx} 
\usepackage{amssymb,latexsym,times}
\newcommand{\psfrag}[2]{}
\usepackage{makeidx,proof}
\usepackage{doi}
\usepackage{enumerate}
\newtheorem{theorem}{Theorem}
\newtheorem{definition}[theorem]{Definition}
\newtheorem{lemma}[theorem]{Lemma}

\usepackage{graphicx,wasysym}

\makeatletter
\newcommand*{\textoverline}[1]{$\overline{\hbox{#1}}\m@th$}
\makeatother

\title{The Logic of Approximate Dependence}
\author{Jouko V\"a\"an\"anen
\thanks{This paper was written while the author was visiting the Computer Science Department of the
University of California, Santa Cruz. The author is grateful to his host Phokion Kolaitis for the invitation and for the hospitality. The author is grateful for helpful discussions on this topic with P. Galliani, L. Hella,  and P. Kolaitis. The author also thanks J. Kivinen, Jixue Liu, H. Toivonen and M. Warmuth  for helpful suggestions. Research partially supported by
grant 40734 of the Academy of Finland.}\\ 
Department of Mathematics and Statistics\\ University of Helsinki, Finland\\ and \\
Institute for Logic, Language and Computation \\ University of Amsterdam, The Netherlands
}
\begin{document}
\maketitle
\def\vx{\vec{x}}
\def\vxp{\vec{x}\hspace{1pt}'}
\def\vyp{\vec{y}\hspace{1pt}'}
\def\vxpp{\vec{x}\hspace{1pt}''}
\def\vypp{\vec{y}\hspace{1pt}''}
\def\vzp{\vec{z}\hspace{1pt}'}
\def\vzpp{\vec{z}\hspace{1pt}''}
\def\vy{\vec{y}}
\def\vz{\vec{z}}
\def\vu{\vec{u}}
\def\vv{\vec{v}}
\def\vw{\vec{w}}
\def\boto{\ \bot\ }
\def\botop{\ \bot'\ }
\newcommand{\idep}[3]{#1\ \bot_{#2}\ #3}
\newcommand{\idepb}[2]{#1\ \bot\ #2}
\def\yxz{\idep{\vy}{\vx}{\vz}}
\def\xzy{\idep{\vx}{\vz}{\vy}}
\def\xzx{\idep{\vx}{\vz}{\vx}}
\def\xyx{\idep{\vx}{\vy}{\vx}}
\def\xzu{\idep{\vx}{\vz}{\vu}}
\def\uzy{\idep{\vu}{\vz}{\vy}}
\def\yzy{\idep{\vy}{\vz}{\vy}}
\def\yxy{\idep{\vy}{\vx}{\vy}}
\def\xyu{\idep{\vx}{\vy}{\vu}}
\def\xy{\idep{\vx}{\vy}}
\def\dyx{\dep(\vy,\vx)}

\begin{abstract}
In my joint paper \cite{MR2132917} with Rohit Parikh we investigate a logic arising from finite information. Here we consider  another kind of limited information, namely information with a small number of errors, and prove a related completeness theorem. We point out that this approach naturally leads to considering multi-teams in the team semantics that lies behind \cite{MR2132917}.
\end{abstract}

\def\mm{\mathfrak{M}}


The idea of {\em finite information logic} [1] is that when quantifiers, especialy the existential quantifiers, express choices in social context, the choices are based on {\em finite} information about the parameters present. In this paper we consider a different kind of restriction. We do not restrict the information available but we allow a small number of errors. In {\em social sofware} a few errors can perhaps be allowed, especially if there is an agreement about it. 

Consider the sentences
\smallskip

``On these flights I have an exit-row seat, apart from a few exceptions.'',
\smallskip

``Apart from a few, the participants are logicians.'',
\smallskip

\noindent One way to handle such expressions is the introduction of generalized quantifiers, such as ``few $x$'',
``most $x$'', ``all $x$ but a few'', etc.
The approach of this paper is different, or at least on the surface it looks different. We use the {\em team semantics} of [2].

In team semantics the main tool for defining the meaning of formulas is not that of an assignment but that of a {\em set} of assignments. Such sets are called in [2] teams. Intuitively a team can represent (or manifest) different kinds of things, such as

\begin{center}
 
\begin{tabular}{l}
 uncertainty\\
belief\\
plays of a game\\
data about a scientific experiment\\
possible voting profiles\\
database\\
dependence\\
independence \\
inclusion\\
exclusion\\
etc. 
\end{tabular}
\end{center}

The obvious advantage of considering meaning in terms of teams rather than single assignments is that teams can indeed manifest change and variation unlike static single assignments. For example, Table 1 can tells us e.g. that $y=x^2$ apart from one exception, and $z$ is a constant zero apart from one exception. Likewise, Table 2 tells us that an employee's salary depends only on the department except for one person. In medical data about causes and effects of treatment there can often be exceptions although there may be compelling evidence of a causal relationship otherwise.

\begin{table}
\centering
\begin{tabular}{l|c|r}
$x$&$y$&$z$\\ \hline
2&4&0\\
5&25&0\\ 
3&9&1\\
2&3&0\\ 
\end{tabular}

\medskip
Table 1
\end{table}

\begin{table}
\centering
\begin{tabular}{l|c|r}
Employee&Department&Salary\\ \hline
John&I&120 000\\
Mary&II&130 000\\ 
Ann&I&120 000\\
Paul&I&120 000\\
Matt&II&130 000\\
Julia&I&130 000\\
\end{tabular}

\medskip
Table 2
\end{table}

While team semantics is suitable for many purposes, we focus here on the concept of dependence, the main concept of the paper [1], too. Dependence is used throughout science and humanities. In particular it appears in database theory in the form of functional dependence.


In \cite{MR2351449} the following concept was introduced:

\begin{definition}
 A {\em team} is any set of assignments for a fixed set of variables.
 A team $X$ is said to satisfy the dependence atom 
\begin{equation}\label{1}
 \dep(x, y),
\end{equation}
where $x$ and $y$ are finite sequences of variables, if any two assignments $s$ and $s'$ in $X$ satisfy
\begin{equation}
s(x)=s'(x) \rightarrow s(y)=s'(y).
\end{equation}
\end{definition}

Dependence logic ([2]) arises from first order logic by the addition of the dependence atoms (\ref{1}). The logical operations $\neg$, $\land$, $\vee$, $\forall$, $\exists$ are defined in such a way that dependence logic is a conservative extension of classical first order logic. The exact expressive power of dependence logic is existential second order logic. 

With the purpose in mind to capture a concept of dependence logic which is more realistic in the sense that a couple of errors are allowed, we now define\footnote{An essentially same, as well as related approximate functional dependences, were introduced already in \cite{Kivinen1995129}.}:

\begin{definition}
  Suppose $p$ is a real number, $0\leq p \leq 1$. A finite team $X$ is said to satisfy the {\em approximate dependence atom}

$$\depp{p}(x, y)$$

\noindent if there is $Y \subseteq X$, $|Y|\leq p \cdot |X|$, such that the team $X\setminus Y$ satisfies $\dep(x, y)$. We then write $X\models\depp{p}(x, y)$. For arbitrary teams (finite or infinite) $X$ we say that $X$ satisfies the atom $\dep(x, y)$ {\em mod finite}, if there is a finite $Y$ such that $X\setminus{Y}$ satisfies $\dep(x, y)$. In symbols $X\models \dep^*(x, y)$.
\end{definition}

In other words, a finite team of size $n$ satisfies $\depp{p}(x,y)$ if it satisfies $\dep(x,y)$ after we delete a portion measured by the number $p$, of the assignments of $X$. More exactly, we delete up to  $p\cdot n$ assignments from the team. Hence the word ``approximate".

The emphasis in approximate dependence $\depp{p}(x, y)$ is on small $p$ but the general concept is defined for all $p$. It is difficult to imagine any practical importance for, say, $\depp{.95}(x, y)$. This is the proposition that the team has a 5\% portion where $y$ is functionally determined by $x$. However, if we suppose that the relationship between $x$ and $y$ is totally random, then it may be significant in a big dataset of millions of rows, to observe that the 
$\depp{.95}(x, y)$ holds and violates total randomness.

For a trivial example, every finite team satisfies $\depp{1}(x, y)$, because the empty team always satisfies $\dep(x, y)$. On the other hand, $\depp{0}(x, y)$ is just the old $\dep(x, y)$. Since singleton teams always satisfy $\dep(x,y)$, a team of size $n$ always satisfies $\depp{1-\frac{1}{n}}(x,y)$.
 A finite team trivially satisfies $\dep^*(x, y)$, whatever $x$ and $y$, so the ``mod finite'' dependence is only interesting in infinite teams. 

The team of Table 1 satisfies $\depp{\frac{1}{4}}(x, y)$ and the team of Table 2 satisfies 

$$\depp{\frac{1}{6}}({\tt Department}, {\tt Salary}).$$

We claim that approximate dependence $\depp{p}(x, y)$ is a much more common phenomenon in science and humanities than full dependence $\dep(x, y)$. Any data\-base of a significant size contains errors for merely human reasons or for reasons of errors in transmission. Any statistical data of medical, biological, social, etc information has exceptions partly because of the nature of the data. One rarely if ever encounters absolute dependence of the kind $\dep(x, y)$ in practical examples. The dependencies we encounter in practical applications have exceptions, the bigger the data the more there are exceptions. For the dependence $\depp{.1}(x, y)$ we allow an error value in 10\% of the cases. This may be unacceptable for some applications but overwhelming evidence of functional dependence in others.

A different kind of approximate functional dependence arises if we think of the individual {\em values} of variables as being slightly off. For example, we can consider a functional dependence in which values of $y$ are {\em almost} the same whenever the values of $x$ are almost the same. This direction is pursued in \cite{brvv}.

We have emphasized the relevance of $\depp{p}(x, y)$ over and above $\dep(x, y)$. So how does dependence logic change if we allow $\depp{p}(x, y)$ in addition to $\dep(x, y)$, that is, if we allow dependance with errors in addition to dependence without errors?

One of the first results about database dependencies is the so called Armstrong Completeness Theorem [3]. It has as its starting point a set of axioms for the dependance $\dep(x, y)$. We now adapt the axioms from [3] to the more general case of approximate dependence atoms. Concatenation of two finite sequences of variables, $x$ and $y$, is denoted $xy$. Such finite sequences can be also empty.

\begin{definition}\label{2}
The {\em axioms} of  approximate dependence are:
\begin{description}
\item[A1] $\depp{0}(xy,x)$ (Reflexivity)
\item[A2] $\depp{1}(x,y)$ (Totality)
\end{description}
The {\em rules} of  approximate dependence are:
\begin{description}
\item[A3] If $\depp{p}(x,yv)$, then $\depp{p}(xu,y)$ (Weakening)
\item[A4] If $\depp{p}(x,y)$, then $\depp{p}(xu,yu)$ (Augmentation)
\item[A5] If $\depp{p}(xu,yv)$, then $\depp{p}(ux,yv)$ and $\depp{p}(xu,vy)$ (Permutation)
\item[A6] If $\depp{p}(x,y)$ and $\depp{q}(y,v)$, where $p+q\le 1$, then $\depp{p+q}(x,v)$  (Transitivity)
\item[A7] If $\depp{p}(x,y)$ and $p\le q\le 1$, then $\depp{q}(x,y)$ (Monotonicity)
\end{description}
 
\end{definition}

These axioms are always satisfied in finite teams. As to (A1), we observe that the empty team $\emptyset$ satisfies $\dep(x, y)$ and hence we can take ${Y}=X$ in Definition~\ref{2}. As to (A2) we observe that every team satisfies $\dep(x, y)$ and so we can take ${Y}=\emptyset$ in Definition~\ref{2}. The axiom (A3) can be verified as follows. Suppose $X\setminus{Y} \models \dep(x, yz)$, where ${Y}\leq p\cdot X$, and the domain of $X$ (and of ${Y}$) includes $xuyz$ so that both $\dep(x, yz)$ and $\dep(xu, y)$ can be meaningfully checked for satisfiability in $X$. Suppose $s$, $s'\in X$\textbackslash${Y}$ such that $s(xu)=s'(xu)$. Then $s(x)=s'(x)$. 
Hence $s(yz)=s'(yz)$, whence finally $s(y)=s'(y)$. Let us then verify the validity of (A6). Suppose $X\setminus{Y} \models \dep(x, y)$, $X\setminus{Z} \models \dep(y, z)$,     where $|{Y}|\leq p\cdot|X|$ and $|Z|\leq q\cdot|X|$. Then $|{Y}\cup Z|\leq|{Y}|+|Z|\leq (p+q)\cdot |X|$ and $X\setminus ({Y}\cup Z)\models \dep(x, z)$. Finally, (A7) is trivial.

The above axioms and rules are designed with finite derivations in mind. With infinitely many numbers $p$ we can have infinitary logical consequences (in finite teams), such as 
$$\mbox{$\{\depp{\frac{1}{n}}(x,y) : n=1,2,\ldots\}\models\depp{0}(x,y)$},$$ which do not follow by the axioms and rules (A1)-(A6)\footnote{We can use this example to encode the Halting Problem to the question whether a recursive set of approximate dependence atoms logically implies a given approximate dependence atom.}. We now focus on finite derivations and finite sets of approximate dependences. We prove the following Completeness Theorem\footnote{Proposition A.3 of \cite{Kivinen1995129} is a kind of completeness theorem in the spirit the below theorem for one-step derivations involving approximate dependence atoms.}:

\begin{theorem}\label{main}
Suppose $\Sigma$ is a finite set of approximate dependence atoms. Then $\depp{p}{(x,y)}$ follows from $\Sigma$ by the above axioms and rules if and only if every finite team satisfying $\Sigma$ also satisfies $\depp{p}(x,y)$. 
\end{theorem}

We first develop some auxiliary concepts and observations for the proof.
 
Let $\tau$ be a pair $(\Sigma,\depp{p}(x,y))$, where  $\Sigma$ is a finite set of approximate dependencies. For such $\tau$ let $Z_\tau$ be the finite set of all variables in $\Sigma\cup\{\depp{p}(x,y)\}$. 
Let $C_\tau$ be the smallest set containing $\Sigma$ and closed under the rules $(A1)-(A6)$ (but not necessarily under (A7)) for variables in $Z_\tau$. 
Note that $C_\tau$ is finite.

\begin{lemma}\label{3}
$\Sigma\vdash\depp{t}(u,v)$ iff $\exists r\le t(\depp{r}(u,v)\in C_\tau)$. 
\end{lemma}

\begin{proof}
The implication from right to left is trivial. For the converse it suffices to show that the set $$\Sigma'=\{\depp{t}(u,v) : \exists r\le t(\depp{r}(u,v)\in C_\tau)\}$$ is closed under (A1)-(A7). 
\end{proof}

\begin{definition}
 Suppose $\tau=(\Sigma,\depp{p}(x,y))$. For any  variable $y$ let 
$$d_\tau(y)=\min\{r\in[0,1]:\depp{r}(x,y)\in C_\tau\}.$$
 
\end{definition}

This definition
makes sense because there are only finitely many  $\depp{r}(u,v)$ in $C_\tau$. Note that $d_\tau(x)=0$ by axiom (A1). By Lemma~\ref{3},
$$d_\tau(y)=\min\{r\in[0,1]:\Sigma\vdash\depp{r}(x,y)\}.$$

 

\begin{lemma}
If $\Sigma\vdash\depp{p}(u,v)$, then $d_\tau(v)-d_\tau(u)\le p$.  
\end{lemma}

\begin{proof} 
Suppose $d_\tau(u)=r$, $d_\tau(v)=t$, $\Sigma\vdash\depp{r}(x,u)$ ($r$ minimal) and $\Sigma\vdash\depp{t}(x,v)$ ($t$ minimal). 
Now $\Sigma \vdash \depp{r}(x,u)$ and $\Sigma\vdash\depp{p}(u,v)$. Hence $\Sigma\vdash\depp{r+p}(x,v)$. By the minimality of $t$, $t\le r+p$. Hence $t-r\le p$.

\end{proof}

For a given $\Sigma$ there are only finitely many numbers $d_\tau(u)$, $u\in Z_\tau$, because $C_\tau$ is finite. Let $A_\tau$ consist of $p$ and the set of $d_\tau(u)$ such that $u\in Z_\tau$.
Let  $n=1+\max\{\lceil 2/(a-b)\rceil :a,b\in A-\tau, a\ne b\}$.
We define a team $X_\tau$ of size $n$ as follows:
$$X_\tau=\{s_0,\ldots, s_n\},$$
where for $\frac{m}{n}\le d_\tau(u)<\frac{m+1}{n}$ we let
$$s_i(u)=\left\{
\begin{array}{ll}
i, &\mbox{if } {i}\le m \\
m, &\mbox{if }{i}> m\\
\end{array}\right.
$$

\begin{figure}
\begin{center}
 
\begin{tabular}{c|cccc}
     &$x$&\ldots&$u$&\ldots\\
\hline 
$s_0$&0&\ldots& 0  & \ldots      \\
$s_1$&0&\ldots& 1  & \ldots      \\
$s_2$&0&\ldots& 2  & \ldots      \\
\vdots\\
$s_{m}$&0&\ldots& $m$  & \ldots      \\
\vdots\\
$s_{n-1}$&0&\ldots& $m$  & \ldots      \\
\end{tabular}

\end{center}
\caption{The team $X_\tau$} 
\end{figure}

\begin{lemma}\label{fre}
Suppose $X_\tau\models \depp{p}(x,y)$. Then $\Sigma\vdash\depp{p}(x,y)$. 
\end{lemma}

\begin{proof} Suppose $X_\tau\models \depp{p}(x,y)$ but $\Sigma\nvdash\depp{p}(x,y)$.
Now $d_\tau(y)>p$. 
Let $\frac{m}{n}\le d_\tau(y)<\frac{m+1}{n}$. One has to take all the assignments $s_i$, $i\le m-1$, away from $X_\tau$ in order for the remainder to satisfy $\dep(x,y)$. Hence $p\cdot n\ge m$ i.e. $p\ge \frac{m}{n}$. But we have chosen $n$ so that $1/n<d_\tau(y)-p$. Hence $$p<d_\tau(y)-\frac{1}{n}\le \frac{m+1}{n}-\frac{1}{n}=\frac{m}{n},$$ a contradiction.
\end{proof}

\begin{lemma}\label{fer}
Suppose $\Sigma\vdash\depp{q}(u,v)$. Then $X_\tau\models \depp{q}(u,v)$. 
\end{lemma}

\begin{proof}
We know already $d_\tau(v)-d_\tau(u)\le q$. If $d_\tau(v)\le d_\tau(u)$, then $X_\tau\models\dep(u,v)$, and hence all the more $X_\tau\models \depp{q}(u,v)$. Let us therefore assume $d_\tau(v)>d_\tau(u)$. Since $2/n<d_\tau(v)-d_\tau(u)$, there are $m$ and $k$ such that $$\frac{m}{n}\le d_\tau(u)<\frac{m+1}{n}<\frac{k}{n}\le d_\tau(v)<\frac{k+1}{n}.$$ 

In order to satisfy $\dep(x,y)$ one has to delete $k-m$ assignments from $X_\tau$. But this is fine, as $qn\ge (d_\tau(v)-d_\tau(u))n\ge k-d_\tau(u)n\ge k-m$. 

 

\end{proof}

Lemmas \ref{fre} and \ref{fer} finish the proof of Theorem~\ref{main}.
\medskip

A problematic feature of the approximate dependence atom is that it is  {\bf not local}, that is, the truth of $X\models\depp{p}(x,y)$ may depend on the values of the assignments in $X$ on variables $u$ not occurring in $x$ or $y$.
To see this, consider the team $Y$ of Figure~\ref{nl}. Now $Y$ satisfies $\depp{\frac{1}{3}}(x,y)$. Let $Z$ be the team $Y\restriction xy$.
 Now $Z$ does not satisfy $\depp{\frac{1}{3}}(x,y)$, as Figure~\ref{nl} shows.

\begin{figure}
\begin{center}
 
\begin{tabular}{ccc}
$x$&$y$&$z$\\
\hline
0&0&0\\
0&0&1\\
0&1&1\\
\multicolumn{3}{c}{$Y$}\\
\end{tabular}
\qquad \begin{tabular}{cc}
$x$&$y$\\
\hline
0&0\\
0&1\\
\\
\multicolumn{2}{c}{$Y\restriction xy$}\\
\end{tabular}

\end{center}
\caption{Non-locality\label{nl} of approximation} 
\end{figure}

This problem can be overcome by the introduction of {\bf multi-teams}:
\begin{definition}
A {\em multi-team} is a pair $(X,\tau)$, where $X$ is a set and $\tau$ is a function such that
\begin{enumerate}
\item   $\dom(\tau)=X$,
\item  If $i\in X$, then  $\tau(i)$ is an assignment for one and the same set of variables. This set of variables is denoted by $\dom(X)$. 
 \end{enumerate}

\end{definition}

An ordinary team $X$ can be thought of as the multi-team $(X,\tau)$, where $\tau(i)=i$ for all $i\in X$. When approximate dependence is developed for multi-teams the non-locality phenomenon disappears (see Figure~\ref{nlv}). Moreover, the above Theorem~\ref{main} still holds. 

The canonical example of a team in dependence logic is the set of plays where a player is using a fixed strategy. Such a team satisfies certain dependence atoms reflecting commitments the player has made concerning information he or she is using. If such dependence atoms hold only approximatively, the player is allowed to make a small number of deviations from his or her commitments. Let us suppose the player is committed to $y$ being a function of $x$ during the game. Typically $y$ is a move of this player and $x$ is the information set available for this move.  When we look at a table of plays where the player is following his or her strategy, we may observe that indeed $y$ is functionally determined by $x$ except in a small number of plays. To evaluate the amount of such exceptional plays we can look at the table of all possible plays where the said strategy is used and count the numerical proportion of plays that have to be omitted in order that the promised functional dependence holds.

We have here merely scratched the surface of approximate dependence. When approximate dependence atoms are added to first order logic we can express propositions such as ``the predicate $P$ consists of half of all elements, give or take 5\%" or ``the predicates $P$ and $Q$ have the same number of elements, with a 1 \% margin of error". To preserve locality we have to introduce multi-teams. On the other hand that opens the door to probabilistic teams, teams where every assignment is associated with a probability with which a randomly chosen element of the team is that very assignment. We will not pursue this idea further here.

\begin{figure}
\begin{center}
 
\begin{tabular}{ccc}
$x$&$y$&$z$\\
\hline
0&0&0\\
0&0&1\\
0&1&1\\
\multicolumn{3}{c}{$Y$}\\
\end{tabular}
\qquad \begin{tabular}{cc}
$x$&$y$\\
\hline
0&0\\
0&0\\
0&1\\
\multicolumn{2}{c}{$Y\restriction xy$}\\
\end{tabular}

\end{center}
\caption{Multi-teams\label{nlv}} 
\end{figure}

\def\Dbar{\leavevmode\lower.6ex\hbox to 0pt{\hskip-.23ex \accent"16\hss}D}
  \def\cprime{$'$}

\end{document}